\documentclass[11pt, reqno]{amsart}

\usepackage{amsthm,amssymb,amstext,amscd,amsfonts,amsbsy,amsrefs,amsxtra,latexsym,amsmath,xcolor,mathrsfs,fancybox,upgreek, soul,url}
\usepackage[top=1.8in, bottom=1.8in, left=3.5in, right=3.5in]{geometry}
\usepackage[english]{babel}
\usepackage[all,cmtip]{xy}
\usepackage[latin1]{inputenc}
\usepackage{cancel}
\usepackage[draft]{hyperref}
\usepackage{comment}
\usepackage{mdframed}
\allowdisplaybreaks
\usepackage{mathtools}
\usepackage{enumerate}
\usepackage{thmtools}
\usepackage{thm-restate}
\usepackage{enumitem}
\usepackage{etoolbox}
\usepackage{todonotes}
\usepackage{footmisc}
\usepackage{bigints}

\usepackage{float}
\restylefloat{table}

\DeclarePairedDelimiter\abs{\lvert}{\rvert}%
\DeclarePairedDelimiter\norm{\lVert}{\rVert}%

\makeatletter
\let\oldabs\abs
\def\abs{\@ifstar{\oldabs}{\oldabs*}}
\let\oldnorm\norm
\def\norm{\@ifstar{\oldnorm}{\oldnorm*}}
\makeatother

\newtheorem{theorem}{Theorem}
\newtheorem{lemma}[theorem]{Lemma}
\newtheorem{corollary}[theorem]{Corollary}
\newtheorem{proposition}[theorem]{Proposition}

\newtheorem*{question}{Question}

\theoremstyle{definition}

\theoremstyle{remark}

\newtheorem*{remark}{Remark}

\newtheorem{example}[theorem]{{\bf Example}}
\numberwithin{theorem}{section}
\numberwithin{proposition}{section}
\numberwithin{lemma}{section}
\numberwithin{corollary}{section}
\numberwithin{equation}{section}
\numberwithin{conjecture}{section}
\numberwithin{example}{section}

\setlist[enumerate,1]{before=}
\AfterEndEnvironment{enumerate}{}

\newcommand{\N}{\mathbb{N}}
\newcommand{\Z}{\mathbb{Z}}

\renewcommand{\pmod}[1]{\  \,  \left( \mathrm{mod} \,  #1 \right)}

\renewcommand{\pmod}[1]{\  \,  \left( \mathrm{mod} \,  #1 \right)}

\begin{document}

\title[Combinatorial results on $t$-cores and sums of squares]{Combinatorial results on $t$-cores and sums of squares}
\author{Joshua Males}
\address{450 Machray Hall, Department of Mathematics, University of Manitoba, Winnipeg, Canada}
\email{joshua.males@umanitoba.ca}
\author{Zack Tripp}
\address{Department of Mathematics, 1420 Stevenson Center, Vanderbilt University, Nashville, TN 37240}
\email{zachary.d.tripp@vanderbilt.edu}

\maketitle

\begin{abstract}
	We classify the connection between $t$-cores and self-conjugate $t$-cores to sums of squares. To do so, we provide explicit maps between $t$-core partitions and self-conjugate $t$-core partitions of a positive integer $n$ to representations of certain numbers as sums of squares. For example, the self-conjugate $4$-core partition $\lambda=(4,1,1,1)$ corresponds uniquely to the solution $61=6^2+5^2$. As a corollary, we completely classify the relationship between $t$-cores and Hurwitz class numbers.
	
	Using these tools, we see how certain sets of representations as sums of squares naturally decompose into families of $t$-cores. Finally, we construct an explicit map on partitions to explain the equality $2\operatorname{sc}_7(8n+1) = \operatorname{c}_4(7n+2)$ previously studied by Bringmann, Kane, and the first author.
\end{abstract}

\section{Introduction}
A \emph{partition} $\lambda$ of $n\in\N$ is a non-increasing sequence $\lambda \coloneqq (\lambda_1, \lambda_2, \dots, \lambda_s)$ of non-negative integers $\lambda_j$ such that $\sum_{1\leq j\leq s} \lambda_j = n$.
The \emph{Ferrers--Young diagram} of $\lambda$ is the $s$-rowed diagram

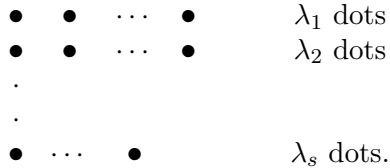
\begin{equation*}
\begin{matrix}
\text{\large $\bullet$} & \text{\large $\bullet$} & \cdots & \text{\large $\bullet$} & \qquad \text{ $\lambda_1$ dots}  \\
\text{\large $\bullet$} & \text{\large $\bullet$} & \cdots & \text{\large $\bullet$} & \qquad \text{ $\lambda_2$ dots}	\\
\cdot \\
\cdot \\
\text{\large $\bullet$} & \cdots & \text{\large $\bullet$} & {}  & \qquad \text{ $\lambda_s$ dots.}
\end{matrix}
\end{equation*}
We label the cells of the Ferrers--Young diagram as if it were a matrix, and let $\lambda_k'$ denote the number of dots in column $k$. The \emph{hook length} of the cell $(j,k)$ in the Ferrers--Young diagram of $\lambda$ equals
\begin{equation*}
h(j,k) \coloneqq \lambda_j + \lambda_k' -k-j+1.
\end{equation*}
If no hook length in any cell of a partition $\lambda$ is divisible by $t$, then $\lambda$ is a \emph{$t$-core partition}. 
A partition $\lambda$ is said to be \emph{self-conjugate} if it remains the same when rows and columns are switched.

\begin{example}
	The partition $\lambda = (3,2,1)$ of $6$ has the Ferrers--Young diagram
	
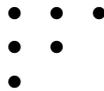
\begin{equation*}
	\begin{matrix}
	\text{\large $\bullet$} & \text{\large $\bullet$} & \text{\large $\bullet$} \\
	\text{\large $\bullet$} & \text{\large $\bullet$} \\
	\text{\large $\bullet$}
	\end{matrix}
	\end{equation*}
	and has hook lengths 
	$h(1,1)=5$, $h(1,2) = 3$, $h(1,3) = 1$, $h(2,1) = 3$, $h(2,2) = 1$, and $h(3,1) = 1$. Therefore, $\lambda$ is a $t$-core partition for all $t \not\in\{ 1,3,5 \}$. Furthermore, switching rows and columns leaves $\lambda$ unaltered, and so $\lambda$ is self-conjugate.
\end{example}

The theory of $t$-core partitions is intricately linked to various areas of number theory and beyond. For example, $t$-core partitions encode the modular representation theory of symmetric groups $S_n$ and $A_n$ (see e.g. \cite{MR1321575,MR671655}). Furthermore, in an influential paper \cite{garvan1990cranks}, Garvan, Kim, and Stanton  used $t$-core partitions to investigate the famous Ramanujan congruences for the partition function $p(n)$, combinatorially proving the special cases given by
\begin{align*}
p(5n+4) \equiv 0 \pmod{5}, \qquad p(7n+5) \equiv 0 \pmod{7}, \qquad p(11n+6) \equiv 0 \pmod{11}.
\end{align*}
 For $t,n \in \N$ we let $\operatorname{c}_t(n)$ denote the number of $t$-core partitions of $n$, along with $\operatorname{sc}_t(n)$ the number of self-conjugate $t$-core partitions of $n$. We also let the set of $t$-cores of $n$ be denoted by $C_t(n)$, and the set of self-conjugate $t$-cores be $SC_t(n)$. In \cite[Theorem 1.1]{Han2010}, Han relates the set of $t$-core partitions for odd $t$ to certain representations as sums of $t$ squares. Han's result builds on \cite{garvan1990cranks}, where Garvan, Kim, and Stanton showed that

\begin{align}\label{Eqn: c_t generating function}
\sum_{n \geq 1} c_t(n) q^n = \sum_{\substack{n \in \Z^t \\ n \cdot 1 = 0}} q^{\frac{t}{2}|n|^2 +b\cdot n}
\end{align}
for $b=[0,1,\dots,t-1]$. Garvan, Kim, and Stanton also showed that \cite[equation (7.4)]{garvan1990cranks}
\begin{align}\label{Eqn: sc_t generating function}
\sum_{n \geq 1} \operatorname{sc}_t(n)q^n = \sum_{n \in \Z^{\lfloor\frac{t}{2} \rfloor}} q^{t |n|^2 + c\cdot n},
\end{align}
where
\begin{align*}
c \coloneqq \begin{cases}
[1,3,\dots,t-1] &\text{ if $t$ is even},\\
[2,4,\dots, t-1] &\text{ if $t$ is odd.}
\end{cases}
\end{align*}
Their proofs relied on constructing an explicit bijection by way of extended $t$-residue diagrams (defined in Section \ref{Section: abaci, t-residue, lists}). With equations \eqref{Eqn: c_t generating function} and \eqref{Eqn: sc_t generating function} in hand, a simple argument via completing the square on the right-hand side yields that for fixed $t$ there is a bijection between $t$-cores and self-conjugate $t$-cores and certain sets of sums of squares. Han proved that for odd $t$ there is a combinatorial connection to representations of certain numbers as sums of squares. To prove this, he showed a bijection between certain normalizations of the $H$-set of a partition and a sum of squares (see \cite{Han2010} for details).

We obtain an explicit combinatorial explanation for all $t$ by considering abaci and the $N$-codings of Garvan, Kim, and Stanton, and thus we relate every $t$-core to a particular sum of squares in the following theorem. Furthermore, we provide related results for all self-conjugate $t$-cores.

We call the set of partitions with fixed number of parts $s$ modulo $t$ a \emph{family} of partitions. The equivalence relation $\sim_{BKM}$ on sets of sums of squares is defined to be permutations and sign changes of variables $x_j$. Recall from \cite{BKM} that when $t$ is odd there are no self-conjugate partitions with $s \equiv \frac{t+1}{2} \pmod{t}$.

\begin{theorem}\label{Thm: main t-cores}
	Let $t \geq 3$. Then there is an explicit bijection between each of the $t$ families of partitions in $C_t(n)$ and a set of solutions, including congruence conditions and a condition on the sum of the variables, to $2tn+ \frac{t(t-1)(2t-1)}{6} + 3(t-1)^2$ as a sum of $t$ squares.  In particular, the set of solutions is
\small	\begin{align*}
	\{ \boldsymbol{x} \in \Z^{t} \colon 2tn+ \frac{t(t-1)(2t-1)}{6} + 3(t-1)^2 =  \sum_{j=0}^{t-1} x_j^2,~ x_j \equiv \pm j \pmod{t}, \sum_{j=0}^{t-1} \pm x_j= \frac{t(t-1)}{2}\}/\sim_{BKM},
	\end{align*}
where the $\pm$ are chosen to have the same sign.
\end{theorem}

\begin{theorem}\label{Thm: SoS - odd t}
	Let $t \geq 3$ be odd. Then there is an explicit bijection between each of the $t-1$ families of partitions in $SC_t(n)$ and a set of solutions, including congruence conditions, to $tn+ \frac{t(t^2-1)}{24}$ as a sum of $\frac{t-1}{2}$ squares. In particular, the set of solutions is
	\begin{align*}
	\{ \boldsymbol{x} \in \Z^{t-1} \colon tn+ \frac{t(t^2-1)}{24} =  \sum_{j=0}^{\frac{t-1}{2}} x_j^2,~ x_j \equiv \pm( j+1) \pmod{t} \}/\sim_{BKM}.
	\end{align*}
\end{theorem}

\begin{theorem}\label{Thm: SoS - even t}
	Let $t \geq 3$ be even. Then there is an explicit bijection between each of the $t$ families of partitions in $SC_t(n)$ and a set of solutions, including congruence conditions, to $4tn+ \frac{t(t^2-1)}{6}$ as a sum of $\frac{t}{2}$ squares. In particular, the set of solutions is 
	\begin{align*}
	\{ \boldsymbol{x} \in \Z^{t} \colon 4tn+\frac{t(t^2-1)}{6} =  \sum_{j=0}^{\frac{t}{2}-1} x_j^2,~ x_j \equiv \pm( 2j+1) \pmod{2t} \}/\sim_{BKM}.
	\end{align*}
\end{theorem}

 For example, a similar phenomenon holds for self-conjugate $4$-cores, where we obtain explicitly that the self-conjugate $4$-core partition $\lambda = (4,1,1,1)$ is uniquely mapped to the solution $61 = 6^2+5^2$, a case which we elucidate in Section \ref{Section: sc4}. Note that this agrees with previous results on $SC_4$ given in \cite{alpoge2014self}, while Theorem 1.4 is more general but states that this particular $4$-core corresponds to the solution $122 = 11^2+1^2$.

The above theorems are proved using the methods of \cite{ono19974, BKM}. However, there is also a natural way to write related results in the language of Han's work \cite{Han2010}. Although the results are similar to the previous theorems, we will present them independently in order to illuminate the deep connection between the results in \cite{ono19974, BKM} and \cite{Han2010}.

\begin{theorem}\label{Thm: alt main t-cores}
	There is an explicit bijection between $C_t(n)$ and certain representations of $8tn + \frac{t(t^2-1)}{3}$ as a sum of $t$ squares, namely
	\begin{equation*}
		\left\{ (w_0, \dots, w_{t-1})\in\Z^t : \sum\limits_{k=0}^{t-1}w_k = 0, \sum\limits_{k=0}^{t-1}w_k^2 = 8tn + \frac{t(t^2-1)}{3}, w_k \equiv 2k+1 - t \pmod{2t} \right\}.
	\end{equation*}
\end{theorem}

Analogous to Han's work, this map will have a simple definition using the $N$-coding of Garvan, Kim, and Stanton \cite{garvan1990cranks}. Using what is known about the $N$-coding of self-conjugate $t$-cores, Theorem~\ref{Thm: alt main t-cores} will allow us to derive a description of self-conjugate $t$-cores as representations of a number into $\lfloor \frac t2 \rfloor$ squares instead.

\begin{theorem}\label{Thm: alt SoS}
	There is an explicit bijection between $SC_t(n)$ and certain representations of $4tn + \frac{t(t^2-1)}{6}$ as a sum of $\lfloor \frac t2 \rfloor$ squares, namely
	\begin{equation*}
		\left\{ (w_0, \dots, w_{\lfloor \frac t2 \rfloor-1})\in\Z^{\lfloor \frac t2 \rfloor} : \sum\limits_{k=0}^{\lfloor \frac t2 \rfloor-1}w_k^2 = 4tn + \frac{t(t^2-1)}{6}, w_k \equiv 2k+1 - t \pmod{2t} \right\}.
	\end{equation*}
\end{theorem}

Aside from these general theorems, we highlight certain cases of sums of squares and their relationships to other objects. As a corollary of Theorem \ref{Thm: SoS - even t}, we complete the picture of self-conjugate $t$-cores and their relationship to class numbers. Combining this with \cite{ono19974,BKM} completely classifies the correspondence between $t$-cores and Hurwitz class numbers. To see this, note that the generating functions for $4$-cores and self-conjugate $6$- and $7$-cores are the only ones that are modular of weight $\frac{3}{2}$ (see \cite{garvan1990cranks}), agreeing with the weight of the generating function of Hurwitz class numbers \cite{zagiereis}.

\begin{corollary}\label{Cor: 6-cores}
	There is an explicit map $\phi$ taking self-conjugate $6$-cores of $n$ to binary quadratic forms of discriminant $-96n-140$. This map does not produce full Hurwitz class numbers.
\end{corollary}
This naturally leads to the following question.
\begin{question}
Are there partition-theoretic objects that the ``missing" quadratic forms arise from in the case of $t=6$?
\end{question}

In the case of $t=9$ consider the set 
\begin{align*}
\mathscr{S}_9 \coloneqq \{ \boldsymbol{x}  \in \Z^4 \colon 9n+30 = x_0^2 +x_1^2 + x_2^2+x_3^2, \boldsymbol{x} \equiv (\pm1, \pm2,\pm3,\pm4) \pmod{9} \}.
\end{align*}
Then \cite[Theorem 10]{alpoge2014self} and Theorem \ref{Thm: SoS - odd t} imply that
\begin{align*}
27 \operatorname{sc}_9(n) & = \frac{27}{16} |\mathscr{S}_9| \\
&= \begin{cases}
\sigma(3n+10) + a_{3n+10} (36a) -a_{3n+10}(54a) - a_{3n+10}(108a)    & \text{ if } n \equiv 1,3 \pmod{4},\\
\sigma(3n+10) + a_{3n+10} (36a) -3 a_{3n+10}(54a) - a_{3n+10}(108a)    & \text{ if } n \equiv 0 \pmod{4}, \\
\sigma(k) + a_{3n+10} (36a) -3 a_{3n+10}(54a) - a_{3n+10}(108a)    & \text{ if } n \equiv 2 \pmod{4},
\end{cases}
\end{align*}
with $\sigma$ the usual sum of divisors, and where $k$ is odd and is defined by $3n+10 = 2^e k$ where $e \in \N_0$ is maximal such that $2^e \mid (3n+10)$. We use Cremona notation for elliptic curves. Here, the $a_n(E)$ are the coefficients appearing in the Dirichlet series for the $L$-function of the elliptic curve $E$. The curve $36a$ is $y^2 = x^3+1$, the curve $54a$ is $y^2+xy = x^3-x^2+12x+8$, and the curve $108a$ is $y^2 = x^3 +4$.

In the course of the paper, we also see that under our maps certain sets of solutions as representations as sums of squares naturally decompose into two sets of partitions. The following example follows immediately from Lemma \ref{Lem: SC_7 C_4}. 
\begin{example}
	 There is a bijection between $C_4(1) \cup SC_7(89)$
	and
	\begin{align*}
	\{(x,y,z) \in \Z^3 \colon x^2+y^2+z^2 = 637 \} / \sim_{BKM}.
	\end{align*}
	 It is clear that $\operatorname{c}_4(1) = 1$ since there is a singular partition of $1$, so $\operatorname{sc}_7(89) = 3$. Let $H(|D|)$ be the Hurwitz class number that counts the number of equivalence classes of positive definite integral binary quadratic forms of discriminant $-D$, and $H_7(|D|)$ the class number that counts the number of equivalence classes of $7$-primitive positive definite integral binary quadratic forms of discriminant $-D$. Then using \cite{BKM} we have
	\begin{align*}
	 \frac{1}{2} H(52)+\frac{1}{4} H_7(2548)  = \frac{1}{48} \Big| \{(x,y,z) \in \Z^3 \colon x^2+y^2+z^2 = 637 \} \Big|,
	\end{align*}
	implying that $H_7(2548) = 12$ and $H(52) = 2$. 
\end{example}
Such considerations also yield inequalities of sets of partitions. Example \ref{Example: sc6 sc7} shows that $\operatorname{sc}_6(7n) \leq \operatorname{sc}_7(24n+3)$, naturally leading to the following question, for example.
\begin{question}
	Is there a way to explicitly realize this inequality as an injection of self-conjugate $6$-cores of $7n$ into self-conjugate $7$-cores of $24n+3$?
\end{question}

Finally, we obtain an explicit map on abaci of partitions to combinatorially prove the equality
\begin{align}\label{Eqn: curious eqn}
\operatorname{c}_4(7n+2) = 2\operatorname{sc}_7(8n+1)
\end{align}
for $n \not\equiv 4 \pmod{7}$ and $56n+21$ square-free, which was shown via class numbers in \cite{BKM}. 
\begin{theorem}
	When $n \not\equiv 4\pmod{7}$, there is an explicit $2$-to-$1$ map $\varphi$ from abaci of $4$-cores of $7n+2$ to abaci of self-conjugate $7$-cores to $8n+1$. The map is invariant under conjugation.
\end{theorem}
While the map will be proved to be $2$-to-$1$ by rewriting maps previously found in \cite{BKM} and \cite{ono19974}, the definition of the map only requires the abacus of the $4$-core. Curiously, the numbers that arise in the definition of this map are directly related to hook lengths of the respective $t$-cores. The reasoning behind this will become clearer in the definition of the map of Theorem~\ref{Thm: alt main t-cores}. 

\subsubsection*{Outline}
In Section \ref{Sec: prelims} we gather preliminary results needed for the rest of the paper. In Section \ref{Sec: SoS} we prove our main theorems of the bijections between families of $t$-cores and sums of squares, and provide examples. Section \ref{Sec: sets of solutions} is dedicated to investigating the decomposition of certain sets of sums of squares into $t$-cores. Finally, in Section \ref{Section: explicit map} we describe the explicit map between $C_4(7n+2)$ and $SC_7(8n+1)$.

\subsubsection*{Acknowledgments}
The authors would like to thank Larry Rolen for helpful discussions and comments on an earlier version of this paper which improved its exposition. We also thank the referees for their valuable comments which improved the paper and for pointing out an error in a previous version of the manuscript. The research of the first author conducted for this paper is supported by the Pacific Institute for the Mathematical Sciences (PIMS). The research and findings may not reflect those of the Institute.
\section{Preliminaries}\label{Sec: prelims}

\subsection{Abaci and $N$-codings}\label{Section: abaci, t-residue, lists}

We next describe the \begin{it}$t$-abacus\end{it} associated to a partition $\lambda$. This 
consists of $s$ beads on $t$ rods constructed in the following way (for more background see \cite{ono19974}). For every $1 \leq j \leq s$ define \emph{structure numbers} by
\begin{equation*}
B_j \coloneqq \lambda_j - j+s.
\end{equation*}
For each $B_j$ there are unique integers $(r_j,c_j)$ such that
\begin{equation*}
B_j = t(r_j-1) +c_j,
\end{equation*}
and $0 \leq c_j < t-1$. The \emph{abacus} for the partition $\lambda$ is then formed by placing one bead for each $B_j$ in row $r_j$ and column $c_j$. Using this construction, \cite[Theorem 4]{ono19974} reads as follows.

\begin{theorem}\label{thm:t-coreabacus}
	Let $A$ be an abacus for a partition $\lambda$, and let $m_j$ denote the number of beads in column $j$. Then $\lambda$ is a $t$-core partition if and only if the $m_j$ beads in column $j$ are the beads in positions
	$
	(1,j), (2,j), \dots, (m_j,j).
	$
\end{theorem}

This means that the abacus for a $t$-core partitions may be represented by $(a_1,a_2, \dots, a_t)$. Furthermore, a direct generalisation of a result of Ono and Sze \cite{ono19974} yields that $C_t(n)$ is in one-to-one correspondence with all abaci of the shape $(0,a_1, \dots, a_{t-1})$. 

\begin{example}
	To demonstrate this for clarity for the reader, we borrow the example of $4$-cores from \cite[page 8]{ono19974}. Each $4$-core may be represented by a $4$-tuple, which after applying \cite[Lemma 1]{ono19974} repeatedly can be written in the form $(0,a_1,a_2,a_3)$. For an abacus of this shape, the bead in the upper-left hand corner naturally corresponds to the smallest part in the partition. In this case, it is clearly $1,2$ or $3$ as these are the only possible values represented by beads in the position $(1,1)$, $(1,2)$ or $(1,3)$. Then, as Ono and Sze note, it is clear that there is a unique abacus of this shape for each $4$-core.
\end{example}

The {\it extended $t$-residue diagram} associated to a $t$-core partition $\lambda$ is constructed as follows (see \cite[page 3]{garvan1990cranks}).
Label a cell in the $j$-th row and $k$-th column of the Ferrers--Young diagram of $\lambda$ by $k-j \pmod{t}$. We also label the cells in column $0$ in the same way. A cell is called {\it exposed} if it is at the end of a row. The {\it region $r$} of the extended $t$-residue diagram of $\lambda$ is the set of cells $(j,k)$ satisfying $t(r-1) \leq k-j < tr$. Then we define $n_j$ to be the maximum region of $\lambda$ which contains an exposed cell labeled $j$. As noted in \cite{garvan1990cranks}, this is well-defined since column $0$ contains infinitely many exposed cells. Using extended $t$-residue diagrams, the authors of \cite{garvan1990cranks} showed the following result. 

\begin{lemma}[Bijection 2 of \cite{garvan1990cranks}]\label{Lemma: Garvan size of lists}
	Let $C_t(n)$ be the set of $t$-core partitions of $n$. There is a bijection $C_t(n) \rightarrow \{ N \coloneqq [n_0, \dots, n_{t-1}] \colon n_j \in \Z, n_0 + \dots + n_{t-1} = 0 \}$ such that
	\begin{equation*}
	|\lambda| = \frac{t|N|^2}{2} + B \cdot N, \hspace{20pt} B \coloneqq [0,1, \dots, t-1].
	\end{equation*}
	When computing the norm and dot-product, we consider $N,B$ as elements in $\Z^t$.
\end{lemma}

\begin{example}
	As a brief example of this bijection, consider the set $C_2(n)$. Then we obtain a bijection between
	\begin{align*}
		C_2(n) \rightarrow \left\{  n_0, n_1 \in \Z \colon  n_0 = -n_1  \right\}
	\end{align*}
such that the partition $\lambda$ has size
\begin{align*}
		|\lambda| = n_0^2 +n_1^2 + n_1.
\end{align*}
\end{example}
Following Han \cite{Han2010}, we call the list $[n_0, \dots, n_{t-1}]$ the \textit{$N$-coding} associated to $\lambda$. We will utilize the fact that if $\lambda$ has $N$-coding $[n_0, \dots, n_{t-1}]$, then its conjugate has $N$-coding $[-n_{t-1}, \dots, -n_0]$ (see the proof of Bijection 2 in \cite{garvan1990cranks}). This means that an alternative characterization of self-conjugate $t$-cores is that they satisfy the equation $n_k = -n_{t-k-1}$ for $0 \le k \le t-1$. 

In previous work by Ono and Sze \cite{ono19974} and by Bringmann, Kane, and the first author \cite{BKM}, explicit bijections from $t$-cores and self-conjugate $t$-cores were given. Before we state these maps, we define 
\begin{align*}
	K^{OS}(n)&:= \{ (x,y,z) \in \Z^3 : x^2 + y^2 + z^2 = n \}/\sim_{OS} \quad \text{and}\\
	K^{BKM}(n)&:= \{ (x,y,z) \in \Z^3 : x^2 + y^2 + z^2 = n \}/\sim_{BKM},
\end{align*}
where two triples are equivalent under $\sim_{OS}$ if they are equal up to reordering the terms and up to two simultaneous sign changes and where two triples are equivalent under $\sim_{BKM}$ if they are equal up to reordering the terms and any number of sign changes. For example, $(x,y,z) \not\sim_{OS} (-x,y,z)$ if $x \neq 0$, while $(x,y,z) \sim_{BKM} (-x,y,z)$.\\
\indent We now recall that Ono and Sze define the partitions $I(g,C,D)$, $II(g,C,D)$, and $III(g,C,D)$ to be the partitions with abaci $(0,g,C+g,D+g)$, $(0,D+g+1, g, D+g)$, and $(0,C+g+1,D+g+1,g)$ respectively for $g,C,D\ge 0$. This describes all possible abaci of $4$-cores, so the bijection $\psi\colon C_4(n) \to K^{OS}(8n+5)$ can be defined by Table \ref{PsiOnoSze}.
\begin{center}
	\begin{table}[H]
		\begin{tabular}{c c c}
			Type of Partition & Shape of Abaci \\
			I(g,C,D) & $(2C - 2D - 2g - 1, 2C-2D+2g, 2C+2D+2g+2)$ \\
			II(g,C,D)& $(2C+2D+2g+3, 2C-2D+2g, 2C-2D-2g-2)$ \\
			III(g,C,D) & $(2C-2D+2g+1, 2C+2D+2g+4, 2C-2D-2g-2)$ \\
		\end{tabular}
		\caption{\label{PsiOnoSze} The different types of abaci for $4$-cores.}
	\end{table}
\end{center}

Similarly, Bringmann, Kane, and the first author classified self-conjugate $7$-cores and give a bijection $\rho\colon SC_7(n) \to K^{BKM}(7n+14)$ as follows:
\begin{center}
	\begin{table}[H]
		\begin{tabular}{c c c}
			Type of Partition & Shape of Abaci & Element of $K^{BKM}(7n+14)$\\
			I & $(0,a,b,r,2r-b,2r-a,2r)$ & $(7r+3, 7r+2-7a, 7r+1-7b)$\\
			II & $(0,2r+1, a,b,r,2r-b,2r-a)$ & $(7r+4,7r+2-7a,7r+1-7b)$\\
			III & $(0,a,2r+1-a,2r+1,b,r,2r-b)$ & $(7r+5,7r+4-7a,7r+1-7b)$\\
			IV & $(0,a,b,2r+1-b,2r+1-a,2r+1,r)$ & $(7r+6,7r+5-7a,7r+4-7b)$\\
			V & $(0,r+1,2r+2,a,b,2r+1-b,2r+1-a)$ & $(7r+8,7r+5-7a,7r+4-7b)$\\
			VI & $(0,a,r+1,2r+2-a,2r+2,2r+1-b)$ & $(7r+9,7r+8-7a,7r+4-7b)$
		\end{tabular}
		\caption{\label{Table: map of BKM} The different types of abaci for self-conjugate $7$-cores and their image under $\rho$.}
	\end{table}
\end{center}

We also make extensive use of the following result, which is \cite[Proposition 4.3]{BKM}.
\begin{proposition}\label{Proposition: list to abacus}
	Let $N = [n_0,\dots, n_{t-1}]$ be the list associated to the extended $t$-residue diagram of a $t$-core partition $\Lambda$. Let $\ell + s = \alpha_\ell t + \beta_\ell$ with $0 \leq \beta_\ell \leq t-1$.
	Then $N$ also uniquely represents the abacus 
	$
	( \dots,n_{t-1} +\alpha_{t-1}, n_0 + \alpha_0, n_1 + \alpha_1, \dots ),
	$
	where $n_\ell+\alpha_\ell$ occurs in position $\beta_\ell$ of the abacus.
\end{proposition}
\begin{example}
	As a small example of Proposition \ref{Proposition: list to abacus}, we borrow the example of \cite{BKM}. Let $t=4$ and construct the abacus and $4$-residue diagram for the partition $\Lambda = (3,2,1)$. 
	We begin with the abacus, computing the structure numbers $B_1 = 5$, $B_2 = 3$, and $B_3 = 1$. Then diagrammatically the abacus is 
	\begin{equation*}
		\begin{matrix}
			{}\vphantom{\begin{smallmatrix}a\\ a\end{smallmatrix}} & \text{\large{$0$}} &  \text{\large{$1$}} & \text{\large{$2$}} &\text{\large{$3$}}  \\
			\text{\large{$1$}} \vphantom{\begin{matrix}a\\ a\end{matrix}}& {} & \text{\large{$B_3$}} & {} & \text{\large{$B_2$}}   \\
			\text{\large{$2$}}\vphantom{\begin{matrix}a\\ a\end{matrix}} & {} & \text{\large{$B_1$}} &  & 	
		\end{matrix}
	\end{equation*}
	The extended $4$-residue diagram of the partition is
	\begin{equation*}
		\begin{matrix}
			{} & \text{\large{$0$}}\vphantom{\begin{smallmatrix}a\\ a\end{smallmatrix}} & \text{\large{$1$}} & \text{\large{$2$}} & \text{\large{$3$}} \\
			\text{\large{$1$}}\vphantom{\begin{matrix}a\\ a\end{matrix}} & {}_3 & \text{\large $\bullet$}_0 & \text{\large $\bullet$}_1 & \text{\large $\bullet$}_2  \\
			\text{\large{$2$}}\vphantom{\begin{matrix}a\\ a\end{matrix}}  & {}_2 & \text{\large $\bullet$}_3 & \text{\large $\bullet$}_0 & {} 	\\
			\text{\large{$3$}}\vphantom{\begin{matrix}a\\ a\end{matrix}} & {}_1 & \text{\large $\bullet$}_2 & {} & 	{} \\
		\end{matrix}
	\end{equation*}
	Then the exposed cells in this diagram are $(1,3)$, $(2,2)$, and $(3,1)$. One may obtain the elements of the list as $n_0 = n_2 =1$ and $n_1=n_3 = -1$ (where the final two values arise from exposed cells in column $0$). Taking this list, an application of Proposition \ref{Proposition: list to abacus} yields the abacus $(0,2,0,1)$, which is precisely the one obtained diagramatically above.
\end{example}

\section{Sums of squares}\label{Sec: SoS}
\subsection{Generic $t$-cores}
In this section we begin by proving Theorem \ref{Thm: main t-cores}. 

\begin{proof}[Proof of Theorem \ref{Thm: main t-cores}]
	Begin by considering the first family of partitions. That is, let $s = tr$. The abacus of a $t$-core partition $\lambda$ is given by 
	\begin{align*}
	(0,a_1,\dots, a_{t-1}),
	\end{align*} 
	where $\sum_j a_j = tr$, and by Proposition \ref{Proposition: list to abacus}, the shape of the $N$-coding associated to $\lambda$ is
	\begin{align*}
	[-r,a_1-r,a_2-r, \dots, a_{t-1}-r].
	\end{align*}
	Using that the sum of the elements in the $N$-coding vanishes by Lemma~\ref{Lemma: Garvan size of lists}, we rewrite $a_{t-1} = tr-\sum_{j=1}^{t-2}a_j$. Lemma \ref{Lemma: Garvan size of lists} relates the size of the partition and the $N$-coding by
	\begin{multline*}
	n = \frac{t}{2}\left( (t-1)r^2 + \sum_{j=1}^{t-2}a_j^2-2a_jr + \left((t-1)r - \sum_{j=1}^{t-2}a_j\right)^2 \right) + \sum_{j=1}^{t-2} j(a_j-r) +(t-1)\left(tr-\sum_{j=1}^{t-2} a_j\right).
	\end{multline*}
	It is not difficult to see that we thus have
	\begin{align*}
	2tn = \sum_{j=1}^{t-2} (tr-ta_j + (t-1 -j))^2 + (tr+2(t-1))^2 - 3(t-1)^2 + \left( (t-1)tr - t\sum_{j=1}^{t-2} a_j \right)^2 - \sum_{j=1}^{t-1}j^2.
	\end{align*}
	
	Identifying the final term as $\frac{t}{6}(t-1)(2t-1)$, we therefore obtain that the partition $\lambda$ is a $t$-core only if $2tn+\frac{t}{6}(t-1)(2t-1) + 3(t-1)^2= \sum_{j=0}^{t-1} x_j^2$ with $x_j \equiv j \pmod{t}$. We also see that $\sum_{j=0}^{t-1}x_j= \frac{t(t-1)}{2}$. A similar calculation holds for other choices of $s \pmod{t}$. Noting that under sign changes of $x_j$ and relabelling of variables we still obtain a representation as a sum of squares, we obtain the statement of the theorem.
\end{proof}

We now prove the related result in Theorem~\ref{Thm: alt main t-cores}. The map will be similar to the map of Theorem~\ref{Thm: main t-cores}. However, this map is inspired by generalizing the work of Han \cite{Han2010}, while the previous result is a generalization of the work of Ono and Sze \cite{ono19974} and Bringmann, Kane, and the first author \cite{BKM}. 

\begin{proof}[Proof of Theorem~\ref{Thm: alt main t-cores}]
	If $W_t(n)$ is the subset of $\Z^t$ defined in the theorem, we may define
	\begin{align}\label{Eq: Han-type map}
		\alpha \colon C_t(n) &\to W_t(n)\notag\\
		[n_k]_{k=0}^{t-1} &\mapsto (2tn_k + 2k+1 -t)_{k=0}^{t-1}=: (w_k)_{k=0}^{t-1},
	\end{align}
	where $[n_k]_{k=0}^{t-1}$ is the $N$-coding of the $t$-core. We must show that the map is well-defined and bijective. In fact, we will show that the conditions of the $N$-coding are equivalent to the conditions in the set $W_t(n)$, proving the bijection. First, because $\sum_{k=0}^{t-1}(2k+1) = t^2$, it is clear that $\sum_{k=0}^{t-1} n_k=0$ is equivalent to $\sum_{k=0}^{t-1} w_k = 0$. The congruence condition $w_k \equiv 2k+1 -t \pmod{2t}$ holds by definition. Finally, we evaluate
	\begin{equation*}
		\sum\limits_{k=0}^{t-1} w_k^2 = \sum\limits_{k=0}^{t-1} (2tn_k +2k+1 - t)^2.
	\end{equation*}
	By expanding the product and using simple sum identities, it is easy to write this as 
	\begin{equation}\label{Eq: SoS}
		8t\left( \frac{t}{2}\sum\limits_{k=0}^{t-1}n_k^2 + \sum\limits_{k=0}^{t-1}kn_k\right) + \frac{t^3 - t}{3} + (4t-2t^2)\sum\limits_{k=0}^{t-1} n_k.
	\end{equation}
	With the properties of an $N$-coding given in Lemma~\ref{Lemma: Garvan size of lists}, we see that \eqref{Eq: SoS} becomes $8tn + \frac{t(t^2-1)}{3}$ as desired. Conversely, given that $\sum_{k=0}^{t-1} w_k^2 = 8tn + \frac{t(t^2-1)}{3}$ and $\sum_{k=0}^{t-1} n_k = \sum_{k=0}^{t-1} w_k = 0$, \eqref{Eq: SoS} shows
	$$\frac{t}{2}\sum\limits_{k=0}^{t-1}n_k^2 + \sum\limits_{k=0}^{t-1} k n_k = n,$$
	i.e. that the $N$-coding comes from a $t$-core of $n$, proving the bijection.
\end{proof}

\subsection{Generic self-conjugate $t$-cores}
It is clear that there are exactly $t-1$ (resp.\@ $t$) families of self-conjugate $t$-cores when $t$ is odd (resp.\@ even). With the same techniques as used in the proof of Theorem~\ref{Thm: main t-cores}, we obtain the following proof of Theorem \ref{Thm: SoS - odd t}.

\begin{proof}[Proof of Theorem \ref{Thm: SoS - odd t}]
	Start with the case that $s=tr$ (recall that $s = \sum_j a_j$). Then by Proposition \ref{Proposition: list to abacus}, the $N$-coding associated to a partition in $SC_t(n)$ has the shape
	\begin{align*}
	[-r,a_1-r,a_2-r,\dots,a_{\frac{t-3}{2}}-r,0,r-a_\frac{t-3}{2}, \dots, r-a_1,r].
	\end{align*}
	Using Lemma \ref{Lemma: Garvan size of lists}, we see that 
	\begin{align*}
	n=t\left(r^2 \left(\frac{t-1}{2}\right) + \sum_{j=1}^{\frac{t-3}{2}} a_j^2 -2a_jr\right) + \sum_{j=1}^{\frac{t-3}{2}} (t-1-2j)(r-a_j) + (t-1)r.
	\end{align*}
	In turn, this leads to 
	\begin{align*}
	tn = \left(tr+\frac{t-1}{2}\right)^2 + \sum_{j=1}^{\frac{t-3}{2}} \left(tr+\frac{t-1-2j}{2}-ta_j\right)^2 - \sum_{j=0}^{\frac{t-3}{2}} \left(\frac{t-1-2j}{2} \right)^2.
	\end{align*}
	Thus identifying the final sum as $\frac{t}{24}(t^2-1)$ we obtain a one-to-one correspondence between this subset of $SC_t(n)$ and representations of $tn+\frac{t}{24}(t^2-1)$ as a sum of $\frac{t-1}{2}$ squares $x_j^2$ ($0\leq j \leq \frac{t-3}{2}$) with each $x_j$ congruent to $\frac{t-1-2j}{2} \pmod{t}$.
	
	It remains to check the remaining cases of $s \pmod{t}$, for which we prove one more case - the rest follow a clear pattern. Next assume that $s \equiv 1 \pmod{t}$. Then the $N$-coding has the shape
	\begin{align*}
	[r+1, a_2-r, \dots, a_{\frac{t-1}{2}} -r,0,r-a_{\frac{t-1}{2}}, \dots, r-a_2,-r-1].
	\end{align*}
	Then Lemma \ref{Lemma: Garvan size of lists} implies that
	\begin{align*}
	n= t\left(\frac{t-1}{2}r^2 +2r+1 +\sum_{j=2}^{\frac{t-1}{2}}a_j^2-2a_jr\right) +\sum_{j=2}^{\frac{t-1}{2}} (r-a_j)(t-2j+1) +(t-1)(-r-1).
	\end{align*}
	Therefore we see that $tn$ is equal to
	\begin{align*}
	\sum_{j=2}^{\frac{t-1}{2}} \left(tr+\frac{t+1-2j}{2} - ta_j \right)^2 + \left(tr+\frac{t+1}{2}\right)^2 +t^2 -t(t-1) -\left(\frac{t+1}{2}\right)^2 - \sum_{j=2}^{\frac{t-1}{2}} \left( \frac{t+1-2j}{2} \right)^2,
	\end{align*}
	which is easily seen to imply that $\lambda$ is a $t$-core partition only if $tn+\frac{t}{24}(t^2-1)$ is a sum of $\frac{t-1}{2}$ squares $x_j^2$, where each $x_j \equiv \frac{t-1-2j}{2} \pmod{t}$ apart from $x_0 \equiv \frac{t+1}{2} \pmod{t}$. The other calculations are similar.
\end{proof}

\begin{proof}[Proof of Theorem \ref{Thm: SoS - even t}]
	The proof for even $t$ is similar to that for odd $t$ and so we only provide the first case. Assume that $s =tr$. Then by Proposition \ref{Proposition: list to abacus}, the first family of partitions has associated $N$-coding
	\begin{align*}
	[-r,a_1-r,\dots,a_{\frac{t}{2} -1} - r, r-a_{\frac{t}{2}-1}, \dots, r-a_1,r].
	\end{align*}
	Using Lemma \ref{Lemma: Garvan size of lists}, we obtain
	\begin{align*}
	n = t\left(\frac{t}{2}r^2 + \sum_{j = 1}^{\frac{t}{2}-1} a_j^2 - 2a_jr \right) + \sum_{j=1}^{\frac{t}{2} - 1} (r-a_j)(t-1-2j) + (t-1)r.
	\end{align*}
	It is easy to see that
	\begin{align*}
	4tn = (2tr+t-1)^2 + \sum_{j=1}^{\frac{t}{2}-1} (2tr + (t-1-2j) -2ta_j)^2 -\sum\limits_{j=1}^{\frac t2 -1}(t-1-2j)^2 - (t-1)^2.
	\end{align*}
	Identifying the final two terms as $\frac{t}{6}(t^2-1)$ we obtain the claim for this case, where each $x_j$ with $0\leq j\leq \frac{t}{2}$ is equivalent to $t-1-2j \pmod{2t}$. The other calculations are again similar.
\end{proof}

We now wish to prove our other result for the connection between self-conjugate $t$-cores and representations as sums of squares, namely Theorem~\ref{Thm: alt SoS}. 

\begin{proof}[Proof of Theorem~\ref{Thm: alt SoS}]
	As noted below Lemma~\ref{Lemma: Garvan size of lists}, a partition is self-conjugate if and only if $n_k = -n_{t-1-k}$ for all $0\le k \le t$. The bijection \eqref{Eq: Han-type map} defines $w_k = 2tn_k + 2k+1-t$, so $w_{t-1-k} = 2tn_{t-1-k} + t-1-2k$. Hence, $n_k = -n_{t-1-k}$ is equivalent to $w_k = -w_{t-1-k}$. The result then follows from Theorem~\ref{Thm: alt main t-cores}.
\end{proof}

\subsection{Examples}
An explicit example of Theorem \ref{Thm: SoS - odd t} is given by \cite{BKM}. Here we describe one simple and one more involved example of Theorem \ref{Thm: SoS - even t} in the cases of $t=4,6$. In the latter case we can relate the output to certain quadratic forms in class groups.

\subsubsection{Self-conjugate $4$-cores}\label{Section: sc4}

By \cite[Theorem 7]{alpoge2014self}, we have
\begin{align}\label{eqn:sc4}
\operatorname{sc}_4(n) = \frac{1}{8} \sharp\{ (x,y) \in \Z^2 \colon x^2+y^2 = 8n+5  \}.
\end{align}
This also has a combinatorial interpretation as follows. Utilizing Proposition \ref{Proposition: list to abacus}, we can determine that the there are four possible shapes of the $N$-coding of self-conjugate $4$-cores.

	\begin{center}
	\begin{table}[H]
		\begin{tabular}{c c}
			{Type of Partition} & {Shape of Associated $N$-coding} \\[5pt]
			I & $[-r,a-r,r-a,r]$ \\[5pt]
			II &  $[r+1,a-r,r-a,-r-1]$\\[5pt] 
			III & $[r+1-a,r+1,-r-1,a-r-1]$ \\[5pt]
			IV & $[a-r,-r-1,r+1,r-a]$ \\[10pt]
		\end{tabular} 
		\caption{\label{Table: list families sc4} The different types of associated $N$-coding for self-conjugate $4$-core partitions.}
	\end{table}
\end{center}

By Lemma~\ref{Lemma: Garvan size of lists}, the size of each type of partition can be related directly to the quantities $r$ and $a$, yielding the following proposition.

\begin{proposition}\label{Prop: lists to quad forms sc4}
	Let $n \in \N$ 
	and $a,r\in\N_0$ be given.
	\begin{enumerate}[wide, labelwidth=!, labelindent=0pt]
		\item[\rm (1)] 
		The Type I partition 
		with parameters $a$ and $r$ 
		is a partition of $n$ 
		if and only if
		\begin{equation*}
		8n+5 =(8r+2-4a)^2 + (4a+1)^2.
		\end{equation*}
		\item[\rm (2)] The Type II partition
		with parameters $a$ and $r$ 
		is a partition of $n$ 
		if and only if 
		\begin{equation*}
		8n+5 = (8r+3-4a)^2 + (4a+2)^2.
		\end{equation*}
		
		\item[\rm (3)] The Type III partition 
		with parameters $a$ and $r$ 
		is a partition of $n$ 
		if and only if 
		\begin{equation*}
		8n+5 = (8r+6-4a)^2 + (4a+1)^2.
		\end{equation*}
		
		\item[\rm (4)] The Type IV partition 
		with parameters $a$ and $r$ 
		is a partition of $n$
		if and only if 
		\begin{equation*}
		8n+5 = (8r+6-4a)^2 + (4a+3)^2.
		\end{equation*}

	\end{enumerate}
\end{proposition}
	
We will omit the proof due to its similarity to the proof below for self-conjugate $6$-cores and because this result can also be found using Proposition 3 and Proposition 1 of \cite{ono19974}. 

\begin{example}
	As an illustrative example, consider the partition $\lambda=(4,1,1,1) \in SC_4(7)$. Computing the structure numbers yields $B_1 = 7, B_2=3, B_3=2, B_4=1$. Thus $\lambda$ corresponds to the abacus $(0,1,1,2)$, which in turn corresponds to the $N$-coding $[-1,0,0,1]$. We can identify this as a type I partition with $r = 1$ and $a = 1$. We then find that $\lambda$ corresponds to the solution 
	\begin{align*}
	61 = 6^2 + 5^2.
	\end{align*}
\end{example}

\subsubsection{Self-conjugate $6$-cores and Corollary \ref{Cor: 6-cores}}
Here we completely describe the families of self-conjugate $6$-cores before using Gauss' map to obtain binary quadratic forms of a certain discriminant.

\begin{lemma}\label{Lemma: conditions on A,B..}
	Assume that $A = (0,a,b,c,d,e)$ is an abacus for a self-conjugate $6$-core partition and recall that $s = a+b+c+d+e$. Let $r\in\N_0$. 
	\begin{enumerate}[wide, labelwidth=!, labelindent=0pt]
		\item[\rm (1)] Assume that $s=6r$. Then
		$
		e=2r,$ $
		a+d=2r,$ $
		b+c = 2r.$
		
		\item[\rm (2)] Assume that $s=6r +1$. Then
		$
		a=2r+1,$ $
		b+e=2r,$ $
		c+d=2r.$
		
		\item[\rm (3)] Assume that $s=6r+2$. Then
		$
		a+b=2r+1,$ $
		c=2r+1,$ $
		d+e=2r.$ 
		
		\item[\rm (4)] Assume that $s=6r+3$. Then
		$
		b+c = 2r+1,$ $
		a+d=2r+1,$ $
		e=2r+1.$ 
		
		\item[\rm (5)] Assume that $s=6r+4$. Then
		$
		c+d=2r+1,$ $
		b+e=2r+1,$ $
		a=2r+2.$ 
		
		\item[\rm (6)] Assume that $s=6r+5$. Then
		$
		d+e=2r+1,$ $
		a+b=2r+2,$ $
		c=2r+2.$
		
	\end{enumerate}
\end{lemma}
\begin{proof}
	We prove (1). By Proposition \ref{Proposition: list to abacus}, we see that $A$ corresponds to the $N$-coding $[-r,a-r,b-r,c-r,d-r,e-r]$. Using \cite[Lemma 4.4]{BKM} and the fact that $s = 6r$, the conditions are easy to determine. The other cases follow in the same way.
\end{proof}

Lemma \ref{Lemma: conditions on A,B..} shows that the abaci of self-conjugate $6$-core partitions naturally fall into one of the distinct families given in Table \ref{Table: acabi families}, enumerated with parameters $a,b,r \in \N_0$. 

\begin{center}
	\begin{table}[H]
		\begin{tabular}{c c}
			{Type of Partition} & {Shape of Abaci} \\[5pt]
			I & $(0,a,b,2r-a,2r-b,2r)$ \\[5pt]
			II &  $	(0,2r+1,a,b,2r-b,2r-a)$\\[5pt] 
			III & $(0,a,2r+1-a,2r+1,b,2r-b)$ \\[5pt]
			IV & $(0,a,b,2r+1-b,2r+1-a,2r+1)$ \\[5pt]
			V &  $(0,2r+2,a,b,2r+1-b,2r+1-a)$\\[5pt]
			VI & $(0,a,2r+2-a,2r+2,b,2r+1-b)$ \\[10pt]
		\end{tabular} 
		\caption{\label{Table: acabi families} The different types of abaci for self-conjugate $6$-core partitions.}
	\end{table}
\end{center}

We relate the families of partitions to quadratic forms, with the relationship shown in the following proposition. For brevity, we write only triples without $\pm$ signs - it is clear that changing the sign on any entry preserves the result.

\begin{proposition}\label{Prop: lists to quad forms}
	Let $n \in \N$ 
	and $a,b,r\in\N_0$ be given.
	\begin{enumerate}[wide, labelwidth=!, labelindent=0pt]
		\item[\rm (1)] 
		The Type I partition 
		with parameters $a$, $b$, and $r$ 
		is a partition of $n$ 
		if and only if
		\begin{equation*}
		24n+35 =(12r+3-12a)^2 + (12r+1-12b)^2 + ( 12r+5)^2.
		\end{equation*}
		\item[\rm (2)] The Type II partition
		with parameters $a$, $b$, and $r$ 
		is a partition of $n$ 
		if and only if 
		\begin{equation*}
		24n+35 = (12r+3-12a)^2 +(12r+1-12b)^2 + (12r+7 )^2.
		\end{equation*}
		
		\item[\rm (3)] The Type III partition 
		with parameters $a$, $b$, and $r$ 
		is a partition of $n$ 
		if and only if 
		\begin{equation*}
		24n+35 = (12r +1-12b)^2 + (12r+7-12a )^2 + (12r+9 )^2.
		\end{equation*}
		
		\item[\rm (4)] The Type IV partition 
		with parameters $a$, $b$, and $r$ 
		is a partition of $n$
		if and only if 
		\begin{equation*}
		24n+35 = (12r+9-12a)^2 + (12r+7-12b)^2 + (12r+11)^2.
		\end{equation*}
		
		\item[\rm (5)] The Type V partition 
		with parameters $a$, $b$, and $r$ 
		is a partition of $n$ 
		if and only if
		\begin{equation*}
		24n+35 = (12r+9-12a)^2 + (12r+7-12b)^2 + (12r+13)^2.
		\end{equation*}

		\item[\rm (6)] The Type VI partition 
		with parameters $a$, $b$, and $r$ 
		is a partition of $n$ 
		if and only if
		\begin{equation*}
		24n+35 = (12r+13-12a  )^2 + (12r+5-12b  )^2 +  (12r+15)^2.
		\end{equation*}
	\end{enumerate}
\end{proposition}
\begin{proof}
	We only prove (1). Combining the definition with Propoisiton \ref{Proposition: list to abacus}, the Type I partition $\lambda$ with parameters $a$, $b$, and $r$ has the associated $N$-coding $[-r,a-r,b-r,r-b,r-a,r]$. By Lemma \ref{Lemma: Garvan size of lists}, we thus have
	\[
	n=|\lambda|=6\left(r^2+(a-r)^2 +(b-r)^2\right) + (a-r)+2(b-r)+3(r-b)+4(r-a)+5r.
	\]
	Hence we 
	see that 
	\begin{align*}
	24n+35 = & 144\left( r^2 + (a-r)^2 +(b-r)^2 \right) + 24\left(a-r + 2(b-r) +3(r-b) +4(r-a) +5r\right) + 35 \\
	=& 144a^2 - 288ar + 144b^2 - 288br + 432r^2 - 72a - 24b + 216r + 35.
	\end{align*}
	This is exactly the expansion of
	\begin{equation*}
	(12r+3-12a)^2 + (12r+1-12b)^2 + (12r+5)^2.
	\end{equation*}
	The other cases follow in the same way, using the associated $N$-coding in Table \ref{Table: list families}.\qedhere
	
	\begin{center}
		\begin{table}[H]
			\begin{tabular}{c c}
				{Type of Partition} & {Shape of Associated $N$-coding} \\[5pt]
				I & $[-r,a-r,b-r,r-b,r-a,r]$ \\[5pt]
				II &  $[r+1,a-r,b-r,r-b,r-a,-r-1]$\\[5pt] 
				III & $[r+1-a,r+1,b-r,r-b,-r-1,a-r-1]$ \\[5pt]
				IV & $[r+1-b, r+1-a, r+1, -r-1, a-r-1, b-r-1]$ \\[5pt]
				V &  $[r+1-b,r+1-a,-r-1,r+1,a-r-1,b-r-1]$\\[5pt]
				VI & $[b-r,-r-1,a-r-1,r+1-a,r+1,r-b]$ \\[10pt]
			\end{tabular} 
			\caption{\label{Table: list families} The different types of associated $N$-coding for self-conjugate $6$-core partitions.}
		\end{table}
	\end{center}
	\rm
\end{proof}

Altogether, this proves the following theorem.
\begin{theorem}\label{Thm: sc6}
	There is a one-to-one correspondence between $SC_6(n)$ and the set
	\begin{align*}
	\mathscr{S}_6 \coloneqq \{ (x,y,z) \in \Z^3 \colon x^2+y^2+z^2 = 24n+35, (x,y,z) \equiv (\pm 1, \pm 3, \pm 5) \pmod{12} \}.
	\end{align*}
\end{theorem}

Note that if $x^2+y^2+z^2 = 24n+35$ then all of $x,y,z$ must be odd: if $2 \mid x$ then we would have $y^2+z^2 \equiv 3 \pmod{4}$ which is impossible. However, there may be triples that are not equivalent to one of the form $(\pm 1, \pm 3, \pm 5)\pmod{7}$ as we will elaborate on below.

In the same way as \cite{BKM} obtained for self-conjugate $7$-cores, by Gauss \cite[article 278]{10.2307/j.ctt1cc2mnd}, for each representation of $24n+35$ as the sum of three squares there corresponds a primitive binary quadratic form of discriminant $-96n-140$. This correspondence is invariant under a pair of simultaneous sign changes on the triple $(x,y,z)$. Explicitly, the correspondence is given by the following. For $(x,y,z) \in \mathscr{S}_6$ let $(m_0, m_1, m_2, n_0,n_1,n_2)$ be an integral solution to
\begin{equation*}
x=m_1n_2 - m_2n_1, \hspace{20pt} y = m_2n_0 - m_0n_2 \hspace{20pt} z = m_0n_1- m_1n_0,
\end{equation*}
where a solution is guaranteed by 
Gauss \cite[article 279]{10.2307/j.ctt1cc2mnd}. 
Then 
\begin{equation}\label{Eqn: quad form gauss}
(m_0u +n_0v)^2 + (m_1u +n_1v)^2 + (m_2u+n_2v)^2
\end{equation}
is a form in $\operatorname{CL}(-96n-140)$. A single sign change produces the inverse to the quadratic form to \eqref{Eqn: quad form gauss}, and so under $\sim$ we identify inverses of quadratic forms. In particular, inverse quadratic forms represent the same integers and so already lie in the same genus.

Further, this map is independent of $(m_0,m_1,m_2,n_0,n_1,n_2)$. Similar to \cite{ono19974,BKM}, we find a map $\phi$ taking self-conjugate $6$-cores $\lambda$ to binary quadratic forms of discriminant $-96n-140$ given by

\begin{equation}\label{Equation: defnintion phi for sc6}
\phi \colon \lambda \rightarrow A \rightarrow N \rightarrow (x,y,z) \rightarrow (m_0,m_1,m_2,n_0,n_1,n_2) \rightarrow \text{binary quadratic form}.
\end{equation}

Although we find an explicit map to binary quadratic forms in the class group, here we cannot obtain class numbers because the set of solutions in Theorem \ref{Thm: sc6} is not complete, contrary to the self-conjugate $7$-core case studied in \cite{BKM}. 

A natural question to pose is: are there partitions that explain the remaining solutions for $24n+35 = x^2 + y^2 + z^2$? For example, we have the solution triple $(5,5,3)$ for $n=1$, which does not arise from a self-conjugate $6$-core, or for $n=4$ the triple $(1,3,11)$. In each of these cases, it is also clear that these triples cannot arise from $4$-cores or self-conjugate $7$-cores either (see Table~\ref{PsiOnoSze} and Table~\ref{Table: map of BKM}). Lemma~\ref{Lem: SC_7 C_4} below gives an example of two different $t$-core sets filling out the entire solution set to a sum of three squares. 

\section{Sets of solutions}\label{Sec: sets of solutions}
\subsection{Generic $SC_{2t}$ and $SC_{2t+1}$}
Here we show that on certain progressions that $SC_{2t}$ and $SC_{2t+1}$ are intricately related.
\begin{lemma}\label{Lem: SC2t SC2t+1}
	We have that 
	\begin{align*}
	SC_{2t}((2t+1)n), \qquad SC_{2t+1}\left(8tn+\frac{t(t-1)}{2}\right)
	\end{align*}
	are governed by representations of $8t(2t+1)n+\frac{t}{3}(4t^2-1)$ as a sum of $t$ squares.
\end{lemma}

\begin{proof}
	 Theorem \ref{Thm: SoS - even t} implies that $ SC_{2t}((2t+1)n)$ is governed precisely by representations of the stated number. Theorem \ref{Thm: SoS - odd t} implies that $SC_{2t+1}\left(8tn+\frac{t(t-1)}{2}\right)$ is also governed by such representations after manipulation.
\end{proof}

\begin{example}\label{Example: sc6 sc7}
	For example, $sc_7$ and $sc_6$ are connected on the progressions $\operatorname{sc}_6(7n)$ and $\operatorname{sc}_7(24n+3)$, where each is governed by the equation $x^2+y^2+z^2 = 168n+35$. Here, the set $SC_7(24n+3)$ covers the whole set of solutions (there are none where $7$ divides $x,y,$ or $z$), and furthermore each of $x,y,z$ must be odd by reducing the equation modulo $4$. Hence we immediately see that the image of $SC_6(7n)$ is a subset of the image of $SC_7(24n+3)$. Since each map is a bijection, we thus have $$\operatorname{sc}_6(7n) \leq \operatorname{sc}_7(24n+3).$$ 
\end{example}

\subsection{$C_4$ and $SC_{7}$}
Here we consider the connections and relationship between $4$-cores and self-conjugate $7$-cores. In this special case, note that $C_4$ also has a representation as a sum of three squares.
	\begin{lemma}\label{Lem: SC_7 C_4}
		 There is a bijection between
		\begin{align*}
		\{(x,y,z) \in \Z^3 \colon x^2+y^2+z^2 = 392n+245 \} / \sim_{BKM}
		\end{align*}
		and the set $\frac{1}{2} C_4(n) \cup SC_7(56n+33)$, where by $\frac{1}{2}C_4(n)$ we mean half of the elements in $C_4(n)$.
	\end{lemma}
	\begin{proof}
		We have from \cite[Corollary 4.8]{BKM} that $SC_7(56n+33)$ corresponds to $$\{(x,y,z) \in \Z^3 \colon x^2+y^2+z^2 = 392n + 245, x,y,z \not\equiv 0 \pmod{7}\} / \sim_{BKM} .$$
		We see that the ``missing" elements are those with one variable divisible by $7$. A simple exercise reducing the equation modulo $7$ shows that these are in fact solutions to
		$$\{(x,y,z) \in \Z^3 \colon x^2+y^2+z^2 = 8n+5\}/ \sim_{BKM}.$$
		By \cite[Proposition 2]{ono19974} this set exactly corresponds to $\frac{1}{2}C_4(n)$.
	\end{proof}

\section{A map between $C_4(7n+2)$ and $SC_7(8n+1)$}\label{Section: explicit map}

We wish to give a combinatorial interpretation to the equation 
\begin{equation*}
	\operatorname{c}_4(7n+2) = 2\operatorname{sc}_7(8n+1)
\end{equation*}
for $n \not\equiv 4 \pmod{7}$. To do so, we define a map $\varphi\colon C_4(7n+2) \to SC_7(8n+1)$. Given a $4$-core $\lambda \vdash 7n+2$ with abacus $(0, a_1, a_2, a_3)$, we let $b_j = 4a_j + j$ for $j = 1, 2, 3$ and reorder the indices $\{1,2,3\}=: \{ j_1, j_2, j_3\}$ so that $b_{j_1} < b_{j_2} < b_{j_3}$. We then consider the numbers
\begin{equation*}
	\mathcal{C} := \left\{ b_{j_2}, b_{j_3}, b_{j_2} - b_{j_1}, b_{j_3} - b_{j_1}, \frac{b_{j_2} + b_{j_3} - b_{j_1}}{2}, b_{j_2} + b_{j_3} - b_{j_1} \right\}. 
\end{equation*}
It will be shown in the course of the proof of Theorem~\ref{combinatorialmap} below that the elements of $\mathcal{C}$ are distinct and non-zero modulo $7$, so we will denote the unique element of $\mathcal{C}$ that is congruent to $i\pmod{7}$ by $c_i$ for $i = 1, \dots, 6$. With this notation in mind, we define
\begin{equation}\label{Def: combinatorialmap}
	\varphi(0,a_1, a_2, a_3) := \left( 0, \left\lfloor \frac{c_1}{7} \right\rfloor, \left\lfloor \frac{c_2}{7}\right\rfloor, \left\lfloor \frac{c_3}{7}\right\rfloor, \left\lfloor \frac{c_4}{7}\right\rfloor, \left\lfloor \frac{c_5}{7}\right\rfloor, \left\lfloor \frac{c_6}{7}\right\rfloor \right).
\end{equation}
\begin{theorem}\label{combinatorialmap}
	For $n \not\equiv 4\pmod{7}$, the map $\varphi$ gives a two-to-one map from $C_4(7n+2)$ to $SC_7(8n+1)$.
\end{theorem}

\begin{remark}
The $b_i$ used to define $\varphi$ are essentially hook lengths of the $4$-core. It is not difficult to check from the definition of the abacus that $b_i$ is the largest structure number congruent to $i\pmod{4}$ plus four (or simply equal to $i$ if no such structure number exists). Similarly, the $c_i$ are also essentially hook lengths of the $7$-core, being equal to the largest structure number congruent to $i\pmod{7}$ plus  seven. This is not a coincidence. The maps of \cite{ono19974, BKM} could be re-written in terms of the map $\alpha$ of Theorem~\ref{Thm: alt main t-cores}, and the numbers $w_k := 2(tn_k + k) - (t-1)$ used to define $\alpha$ are essentially just shifts of these structure numbers. While we could rewrite everything in terms of $\alpha$ to prove Theorem~\ref{combinatorialmap}, we do not find it illuminating to do so.
\end{remark}

To prove this, we will $\varphi$ realize as the composition
\begin{equation}\label{Eqn: composition of maps}
	C_4(7n+2) \xrightarrow{\psi} K^{OS}(56n+21) \xrightarrow{p} K^{BKM}(56n+21) \xrightarrow{\rho^{-1}} SC_7(8n+1),
\end{equation}
where $\rho$ and $\psi$ are defined in Table~\ref{Table: map of BKM} and Table \ref{PsiOnoSze} respectively and where $p$ simply maps a triple to itself. It is easy to check that this is well-defined under the given equivalences $\sim_{OS}$ and $\sim_{BKM}$. Identifying $\varphi$ as the composition \eqref{Eqn: composition of maps} will be sufficient to prove Theorem~\ref{combinatorialmap} because $\psi$ and $\omega$ are known to be bijections (see \cite[Proposition 2]{ono19974} and \cite[Corollary 4.8]{BKM}), while the definitions of $K^{OS}(n)$ and $K^{BKM}(n)$ allow us to see that $p$ is a $2$-to-$1$ map. Hence, the rest of the section will be computing the image of $(0,a_1,a_2,a_3)$ under $\rho^{-1} \circ p \circ \psi$. \\

While we already have a definition for the map $\psi$ given in Table \ref{PsiOnoSze}, we wish to write the map in terms of the numbers $a_1$, $a_2$, and $a_3$ instead. This conveniently will not require us breaking the definition into multiple parts.

\begin{lemma}
	\begin{equation}\label{BeadsUnderPsi}
		\psi(0,a_1, a_2, a_3) = \left( -\frac{b_1 + b_2 - b_3}{2}, \frac{b_1 + b_3 - b_2}{2}, \frac{b_2 + b_3 - b_1}{2} \right) \in K^{OS}(8n+5).
	\end{equation}
\end{lemma}
\begin{proof}
	We will only show this for type $I$ partitions since the proof is similar in other cases. These are the partitions for which $a_1 = g$, $a_2 = C+g$, and $a_3 = D+g$ for $g,C,D \ge 0$. Writing $\psi(0,a_1,a_2,a_3) = (x,y,z)$, Table \ref{PsiOnoSze} tells us that (by possibly reordering terms)
	\begin{align*}
		x+y &= 4(C+g) + 2 = 4a_2 + 2,\\
		y-z &= 4(D+g) + 3 = 4a_3 + 3,\\
		x-z &= 4g + 1 = 4a_1 + 1.
	\end{align*}
	Solving the linear system of equations from here, we find that 
	\begin{equation*}
		(x,y,z) = \left( 2(a_1 + a_2 - a_3), 2(a_2 + a_3 - a_1) +2, 2(a_2 - a_1 - a_3) - 1\right).
	\end{equation*}
	Replacing $a_i$ by $\frac{b_i - i}{4}$, reordering the terms, and making two sign changes, we obtain \eqref{BeadsUnderPsi}.
\end{proof}

We have now found $\psi(0, a_1, a_2, a_3)$, and we already know that $p(x,y,z) = (x,y,z)$. Thus, we only need to compute $\rho^{-1}$. 
\begin{lemma}
	For $(x,y,z)\in K^{BKM}(7n+14)$, suppose without loss of generality that $x = \max(|x|, |y|, |z|)$. If $\{s_1, \dots, s_6\} = \{x, 2x, x\pm y, x\pm z\}$ such that $s_i \equiv i \pmod{7}$, then 
	\begin{equation}\label{rhoinverse}
		\rho^{-1}(x,y,z) = \left(0, \left\lfloor \frac{s_1}{7}\right\rfloor, \left\lfloor \frac{s_2}{7}\right\rfloor, \left\lfloor \frac{s_3}{7}\right\rfloor, \left\lfloor \frac{s_4}{7}\right\rfloor, \left\lfloor \frac{s_5}{7}\right\rfloor, \left\lfloor \frac{s_6}{7}\right\rfloor\right).
	\end{equation}
\end{lemma}
\begin{proof}
	We only prove this for type I self-conjugate $7$-cores since the proof is analogous in the other cases. Let $\lambda$ be such a partition so that its abacus is of the form 
	\begin{equation}\label{Eqn: abacus of self-conjugate 7-core}
		(0,a,b,r,2r-b,2r-a,2r)
	\end{equation}
	for $a,b,r\in\N_0$. Recall from Table~\ref{Table: map of BKM} that $(x,y,z) = \rho(\lambda) = (7r+3, 7r+2-7a, 7r+1-7b)$. At this stage, we cannot immediately conclude that the elements of the triples are equal due to the equivalence relation $\sim_{BKM}$. However, notice that $0 \le a,b, \le 2r$ in order for the entries of the abacus to be non-negative, so $7r+3$ must be the largest element of the triple $\rho(\lambda)$. In other words, by our assumption on $x$, we conclude that $x = 7r+3$. By rearranging the remaining terms and changing signs, we may also say that $y = 7r+2-7a$ and $z = 7r+1-7b$. We may solve for $r$, $a$, and $b$ to find that 
	\begin{equation*}
		r = \frac{x-3}{7}, \quad a = \frac{x-y-1}{7}, \quad b = \frac{x-z-2}{7}.
	\end{equation*}
	Plugging this into \eqref{Eqn: abacus of self-conjugate 7-core} yields
	\begin{equation*}
		\left( 0, \frac{x-y-1}{7}, \frac{x-z-2}{7}, \frac{x-3}{7}, \frac{x+z-4}{7}, \frac{x+y-5}{7}, \frac{2x-6}{7} \right). 
	\end{equation*}
	Using the fact that each entry must be an integer proves \eqref{rhoinverse}.
\end{proof}

To complete the proof of the theorem, notice that if $(x,y,z) = p\circ \psi(0,a_1,a_2,a_3)$ with $x = \max(|x|, |y|, |z|)$, then $x$ must be equal $\frac{b_{j_2} + b_{j_3} - b_{j_1}}{2}$ by the choice of indices $j_1, j_2,j_3$ and by \eqref{BeadsUnderPsi}. By direct computation, we then see $\{x,2x, x\pm y, x\pm z\} = \mathcal{C}$. Hence, by \eqref{rhoinverse}, the image of $(x,y,z)$ under $\rho^{-1}$ becomes \eqref{Def: combinatorialmap}, finishing the proof of Theorem~\ref{combinatorialmap}.

While the above map is explicit, it is not immediately clear which two $4$-cores have the same image. However, it turns out that there is a simple answer: the map $\varphi$ is invariant under conjugation. To see this, we use Proposition 3 of \cite{ono19974}. We will only focus on the first case, which tells us that for $D \ge C$, $I(g,C,D)$ and $I(D-C,C,C+g)$ are conjugate pairs. Notice that by Table \ref{PsiOnoSze}, $I(g,C,D)$ maps to 
$$ (2C-2D-2g-1, 2C-2D+2g,2C+2D+2g+2),$$
while $I(D-C,C,C+g)$ maps to 
$$ (2C-2D-2g-1, -(2C-2D+2g), 2C+2D+2g+2).$$ 
By the definition of $\sim_{BKM}$, these are the same under $p \circ \Psi$ and hence under $\varphi = \rho^{-1} \circ p \circ \psi$. It is easy to see that the remaining cases in Proposition 3 of \cite{ono19974} also map to the same value, proving the invariance of $\varphi$ under conjugation. 

While $\varphi$ may be invariant under conjugation, the existence of self-conjugate $4$-cores implies that this may not explain the fact that $\varphi$ is $2$-to-$1$. We illustrate that this is where our condition $n \not \equiv 4\pmod{7}$ is essential. As we alluded to following Proposition~\ref{Prop: lists to quad forms sc4}, Proposition 3 and Proposition 1 of \cite{ono19974} illustrate that self-conjugate $4$-cores are exactly the $4$-cores whose image under $\psi$ has an element that is $0$, i.e. they are the elements that map to $8n+5$ being represented as a sum of two squares. However, in the definition of $\varphi$, this would mean that we are considering $8(7n+2) + 5 = 56n+21$ as a sum of two squares, and since $56n+21$ is divisible by $7$, it is easy to check that the $56n+21 = x^2 + y^2$ implies that $7 | x,y$. Writing $x = 7x'$ and $y = 7y'$, this in turn implies that $n-4 \equiv 8n+3 \equiv 0 \pmod{7}$. Thus, if $n \not\equiv 4\pmod{7}$, $\operatorname{sc}_4(7n+2) = 0$, so the preimage of an element of $SC_7(8n+1)$ must equal a pair of conjugate elements of $C_4(7n+2)$.


\begin{thebibliography}{99}
		
		\bib{alpoge2014self}{article}{
			author={Alpoge, L.},
			title={Self-conjugate core partitions and modular forms},
			JOURNAL = {J. Number Theory},
			FJOURNAL = {Journal of Number Theory},
			VOLUME = {140},
			YEAR = {2014},
			PAGES = {60--92},
		}


\bib{baruah2014infinite}{article}{
	author={Baruah, N. D.},
	author={Nath, K.},
	TITLE = {Infinite families of arithmetic identities for self-conjugate
		$5$-cores and $7$-cores},
	JOURNAL = {Discrete Math.},
	FJOURNAL = {Discrete Mathematics},
	VOLUME = {321},
	YEAR = {2014},
	PAGES = {57--67},
}




\bib{BKM}{article}{
	author = {Bringmann, K.},
	author= {Kane, B.},
	author={ Males, J.},
TITLE = {On {$t$}-core and self-conjugate {$(2t-1)$}-core partitions in
	arithmetic progressions},
JOURNAL = {J. Combin. Theory Ser. A},
FJOURNAL = {Journal of Combinatorial Theory. Series A},
VOLUME = {183},
YEAR = {2021},
PAGES = {Paper No. 105479, 24},
}

	



		\bib{MR671655}{article}{
			AUTHOR = {Fong, P.},
			author={Srinivasan, B.}
			TITLE = {The blocks of finite general linear and unitary groups},
			JOURNAL = {Invent. Math.},
			VOLUME = {69},
			YEAR = {1982},
			PAGES = {109--153},
		}
		
		
		
		
		
		\bib{garvan1990cranks}{article}{
			author={Garvan, F.},
			author={Kim, D.},
			author={Stanton, D.},
			title={Cranks and t-cores},
			date={1990},
			journal={Invent. Math.},
			volume={101},
			pages={1\ndash 17},
		}
%
		\bib{10.2307/j.ctt1cc2mnd}{book}{
			author={Gauss, C.},
			author={Clarke, A.},
			title={Disquisitiones arithmeticae},
			publisher={Yale University Press},
			date={1965},
			ISBN={9780300094732},
			url={http://www.jstor.org/stable/j.ctt1cc2mnd},
		}
		
		
		\bib{MR1321575}{article}{
			AUTHOR = {Granville, A.},
			author = {Ono, K.}
			TITLE = {Defect zero {$p$}-blocks for finite simple groups},
			JOURNAL = {Trans. Amer. Math. Soc.},
			VOLUME = {348},
			YEAR = {1996},
			PAGES = {331-347},
		}
	
		
		\bib{Han2010}{article}{
			author = {Han, G.},
			title = {The {N}ekrasov-{O}kounkov hook length formula:
				refinement, elementary proof, extension and applications},
			JOURNAL = {Ann. Inst. Fourier (Grenoble)},
			FJOURNAL = {Universit\'{e} de Grenoble. Annales de l'Institut Fourier},
			VOLUME = {60},
			YEAR = {2010},
			NUMBER = {1},
			PAGES = {1--29},
		}
		
		

	
		\bib{OnoRaji}{article}{
		author={Ono, K.},
		author={Raji, W.},
		TITLE = {Class numbers and self-conjugate 7-cores},
		JOURNAL = {J. Combin. Theory Ser. A},
		FJOURNAL = {Journal of Combinatorial Theory. Series A},
		VOLUME = {180},
		YEAR = {2021},
		PAGES = {Paper No. 105427, 8},
	}
		
		\bib{ono19974}{article}{
			author={Ono, K.},
			author={Sze, L.},
			TITLE = {{$4$}-core partitions and class numbers},
			JOURNAL = {Acta Arith.},
			FJOURNAL = {Acta Arithmetica},
			VOLUME = {80},
			YEAR = {1997},
			NUMBER = {3},
			PAGES = {249--272},
		}
		
		
		\bib{zagiereis}{article}{
			author={Zagier, D.},
			title={Nombres de classes et formes modulaires de poids $3/2$},
			language={French, with English summary},
			journal={C. R. Acad. Sci. Paris S\'{e}r. A-B},
			volume={281},
			date={1975},
			number={21},
			pages={Ai, A883--A886},
		}
		






\end{thebibliography}
\end{document}